\documentclass[11pt]{amsart}

\usepackage{amssymb,amsmath,latexsym, amscd, graphicx, url, xypic, hyperref}
\usepackage{color}

\newtheorem{theorem}{Theorem}[section]
\newtheorem{lemma}[theorem]{Lemma}
\newtheorem{proposition}[theorem]{Proposition}
\newtheorem{definition}[theorem]{Definition}

\newtheorem{corollary}[theorem]{Corollary}
\theoremstyle{definition}

\newtheorem{example}[theorem]{Example}

\def\R{\mathbb{R}}

\definecolor{mat}{RGB}{110,180,240}

\begin{document}

\title{Orderable groups and bundles}
\date{\today}

\author[Mathieu Anel and Adam Clay]{Mathieu Anel and Adam Clay}
\address{CIRGET, 
Universit\'e du Qu\'ebec \`a Montr\'eal, 
Case postale 8888, Succursale centre-ville, 
Montr\'eal QC, H3C 3P8} \email{aclay@cirget.ca \\ anelm@uqam.ca}
\urladdr{http://thales.math.uqam.ca/~aclay/ \\ http://thales.math.uqam.ca/~anelm/}

\maketitle

\begin{abstract}
We define what is meant by a strict total order in a category having subobjects, products and fibre products.  This allows us to define the notions of an ordered bundle $X$ and an ordered $G$-set;  when $G=\pi_1(X)$ we relate these structures to orderings of $\pi_1(X)$.   We apply this to prove a theorem of Farrell \cite{Fa} relating right-orderings of $\pi_1(X)$ to embeddings of the universal cover $\widetilde{X} \hookrightarrow \mathbb{R} \times X$, and generalize it by relating bi-orderings of $\pi_1(X)$ to embeddings of the path space $P(X)  \hookrightarrow \mathbb{R} \times X \times X$. 
\end{abstract} 

\section{Introduction}
A group $G$ is {\em right-orderable} if there exists a strict total ordering $<$ of the elements of $G$ such that $g<h$ implies $gf<hf$ for all $f, g, h \in G$. A group $G$ is {\em bi-orderable} if there exists a right-ordering of $G$ that is also invariant under left multiplication, so $g<h$ implies $fg<fh$ for all $f,g,h \in G$.  

There is a well-documented connection between orderability and topology.  A number of recent articles have focused on orderability of the fundamental group of a $3$-manifold and its topological implications \cite{BRW05, BGS, CWatson}.  Another strong direction of research is the investigation of the space of orderings of a group, and the algebraic consequences of its topological structure \cite{Nav, Rivas, NR, Rivas2}.  

In the body of work connecting topology and orderability, a notable outlier is the remarkable work of Farrell \cite{Fa}, which re-characterizes right-orderability of fundamental groups in terms of covering spaces.  While his theorem is very appealing in its statement (see Corollary \ref{cor:farrell}), to the best of our knowledge it has yet to be practically applied in any other topological works. 

This note represents an effort to place Farrell's theorem in a natural topological context.  As a result, our new approach provides a different proof of a more general statement, and allows for the analysis of both right-orderable and bi-orderable groups.


\section{Some categories, notation and background}
\label{notation}

Though left-orderability has become popular amonst topologists, our choice of right-orderability is a consequence of several common topological conventions.  In particular, the definition of concatention of paths $f$ and $g$ as
\begin{equation*}
(f*g)(s)= 
\begin{cases} f(2s) & \text{if $s \in [0, \frac{1}{2}]$,}
\\
g(2s-1) &\text{if $x \in [\frac{1}{2},1]$.}
\end{cases}
\end{equation*}
naturally produces a right action on the fibres of the standard constuction of the universal covering space.

That construction is as follows.  For the remainder of this paper fix a path connected, locally path connected and semilocally simply connected space $B$ with basepoint $b_0$, and set $G = \pi_1(B,b_0)$.  Recall that the universal covering space of $B$ is constructed as

\[ \widetilde{B} = \{ \alpha : [0,1] \rightarrow B | \alpha(0) = b_0\}/ \sim
\]
where the equivalence $\sim$ is homotopy fixing the endpoints.  Thus, an element of $\widetilde{B}$ is an equivalence class $[\alpha]$ of a path $\alpha : [0,1] \rightarrow B $ satisfying $\alpha(0) = b_0$.  The required covering map $p: \widetilde{B} \rightarrow B$ has formula $p([\alpha]) = \alpha(1)$, in $\widetilde{B}$ the basepoint is the equivalence class of the constant path $\alpha(t) = b_0$ for all $t \in [0,1]$.


The space $\widetilde{B}$ admits a left action of the group $G$ as follows.  Given $g \in G$ represented by a loop $\gamma : [0,1] \rightarrow B$ and $\alpha \in \widetilde{B}$, define a left action (by deck transformations) by $g \cdot [\alpha] = [\gamma*\alpha]$.  On the other hand, the fibre $F = p^{-1}(b_0)$ admits a right action defined by $[\alpha] \cdot g = [\alpha * \gamma]$, in fact the fibre (over $b_0$) of any covering space $E \rightarrow B$ admits a right action defined in a similar way using lifts of paths.


Let $\mathbf{Cov}(B)$ denote the category whose objects are covering spaces $p: E \rightarrow B$ of $B$. Given two objects $p_1:E_1 \rightarrow B$ and $p_2:E_2 \rightarrow B$ of $\mathbf{Cov}(B)$, a morphism between them is a continuous map $f : E_1 \rightarrow E_2$ such that the following diagram commutes:

\[
\xymatrix{ 
E_1 
\ar[rr]^{f} 
\ar[dr]_{p_1} 
&& E_2
\ar[dl]^{p_2 } \\ 
& B }
\]

\medskip

A {\em right $G$-set} is a set $X$ equipped with a right action of a group $G$.   Given $x \in X$ we will write the action of $g \in G$ on $x$ as $x \cdot g$. Together with $G$-equivariant maps of sets, $G$-sets form a category $G$-$\mathbf{Set}$. 

\bigskip

Define the so-called fibre functor $F:\mathbf{Cov}(B) \rightarrow G$-$\mathbf{Set}$ by $F(p:E \rightarrow B) = p^{-1}(b_0)$. Then $F$ is the functor that chooses the fibre above $b_0$ for every covering space $p:E \rightarrow B$. The action of $[\gamma]\in G$ on $e\in p^{-1}(b_0)$ is defined as the end of the unique path lifting $\gamma$ in $E$ starting from $e$. On maps $f$ between covering spaces, the functor $F$ acts by restricting $f$ to the fibre above $b_0$. This map is always equivariant for the action of $G$.

On the other side, let $U: G$-$\mathbf{Set} \rightarrow \mathbf{Cov}(B)$ be the functor that associates to a $G$-set $X$ the covering space
\[ U(X) = p:X\times_G\widetilde B \longrightarrow B 
\]
where $X\times_G\widetilde B = (X\times \widetilde{B})/ \sim$ for $([\alpha], x) \sim (g^{-1} \cdot [\alpha], x \cdot g)$ for all $g \in G$ and the map $p$ has formula $p([\alpha], x) = \alpha(1)$.   Here, the set $X$ is topologized with the discrete topology, and the topology on $U(X)$ is a quotient of the product topology.  For maps, define the action of $U$ as follows: $U( f: X \rightarrow X') = h$, where 
\[ h: (X\times \widetilde{B}) / \sim \hspace{1em} \longrightarrow (X'\times \widetilde{B}) / \sim
\]
has formula $h([\alpha], x) = ([\alpha], f(x))$.

\begin{proposition}\label{equivdiscrete}
The functors $F:\mathbf{Cov}(B) \rightarrow G$-$\mathbf{Set}$ 
and $U:G$-$\mathbf{Set}\to \mathbf{Cov}(B)$ 
are inverse equivalences of categories.  
\end{proposition}
\begin{proof}

One can check that $F \circ U \cong \mathrm{Id}_{\mbox{$G$-set}}$ and $U \circ F \cong \mathrm{Id}_{\mathbf{Cov}(B)}$ \cite[Chapter 3, Sections 6 and 8]{PMay}.


\end{proof}

\begin{example}\label{triv}
Let $*$ be the trivial action of $G$ of a one element set, the corresponding covering space is the trivial cover $B \to B$. The object $*$ is the terminal object of the category $G$-$\mathbf{Set}$, and $B\to B$ is terminal in the category of coverings.
\end{example}

\begin{example}\label{exGright}
Let $G_r$ be the right $G$-set defined by $G$ acting on itself by multiplication on the right.
By construction, the fibre of $\widetilde B\to B$ at $b_0$ is $G$ and the right $G$-set corresponding to $\widetilde B$ is $G_r$.
\end{example}

\begin{example}\label{coset}
Let $H$ be a subgroup of $G$, the right action of $G$ on itself defines a right action of $G$ on the set $H \backslash G$ of right $H$-cosets. The corresponding covering space is the quotient $H\backslash \widetilde B$ of $\widetilde B$ by the left action of $H$.
\end{example}

\begin{example}
Let $G_{conj}$ be the right action of $G$ on itself by conjugation: $h \cdot g = g^{-1}hg$ and
let $\widetilde B_{conj}$ be the corresponding covering space.   Its fibre at $b \in B$ is naturally isomorphic to $\pi_1(B,b)$,  so each fibre is equipped with a group structure. In particular, $\widetilde B_{conj}$ has at least one global section given by the unit element in each fibre, so it is not connected.
\end{example}

\begin{example}\label{exGbi}
The product $G\times G$ is naturally the fundamental group of $B\times B$ at $(b_0,b_0)$, so there exists an equivalence between 
covering spaces of $B\times B$ and right $G\times G$-sets.

The path space $P(B)$ of $B$ is constructed in a similar way to the universal covering spaces.  Set
\[ P(B) = \{ \alpha:[0,1] \rightarrow B \} / \sim
\]
where $\sim$ is homotopy fixing the endpoints of the paths.
There's a covering map $p:P(B) \rightarrow B \times B$ given by $p(\alpha) = (\alpha(0), \alpha(1))$.
Observe that the fibre $p^{-1}(b_0, b_0)$ is precisely all those paths which begin and end at the point $b_0$, considered up to homotopy.
Thus there is a canonical identification of the group $G$ with the fibre of $p:P(B) \rightarrow B \times B$ at $(b_0,b_0)$.

For a path $\alpha$ with $p([\alpha]) = (b_0, b_0)$ the natural action of $G \times G$ on $f=[\alpha]$ is given by
\[ f \cdot (g, h) = [\gamma^{-1} * \alpha * \beta] = g^{-1}fh,
\]
where $\gamma, \beta$ are paths representing the fundamental group elements $g, h$.

The corresponding right $G\times G$-set is $_{\ell}G_r$, defined as $G$ with an action by multiplication on the left and on the right: $f \cdot (g,h) = g^{-1}fh$. 

Remark that the pull back of $P(B)$ along $B\to B\times B:b\mapsto (b_0,b)$ is exactly $\widetilde B$.
\end{example}

\bigskip

It will be useful to generalize the equivalence of $\mathbf{Cov}(B)$ and $G$-$\mathbf{Set}$ by considering not only actions of $G$ on sets but also on any topological space. The category $G$-$\mathbf{Space}$ of $G$-spaces is defined as that of topological spaces equipped with a right action of $G$ (continous for the discrete topology on $G$) and $G$-equivariant maps. There is an obvious fully faithful functor 
$$
G\textrm{-}\mathbf{Set}\subset G\textrm{-}\mathbf{Space}.
$$

For $E\to B$ a topological space over $B$, we shall denote by $\widetilde E$ the pull-back of $E$ over $\widetilde B$.
A locally trivial bundle (we shall say simply a bundle) $p:E \rightarrow B$ is {\em locally constant} if $\widetilde E\to \widetilde B$ is a trivial bundle. (Such bundles can be characterized as bundles for which there exists a choice of trivializing charts $\{ (U_i, \phi_i) \}_{i \in I}$ such that the transition functions $t_{ij}: U_i \cap U_j \rightarrow G$ are constant, hence their name.) A map $f:E_1\to E_2$ between two locally constant bundles over $B$ is said to be {\em locally constant} if the induced map $\widetilde f:\widetilde E_1\simeq F_1\times \widetilde B\to \widetilde E_2\simeq F_2\times \widetilde B$ over $\widetilde B$ is constant, i.e. $\widetilde f(x,\tilde b) = (g(x),\tilde b)$ for some continous map $g:F_1\to F_2$.
Let $\mathbf{LCB}(B)$ be the category of locally constant bundles over $B$ and locally constant maps.

If $\mathbf{Top}/B$ is the category of spaces over $B$ with maps over $B$, there is an obvious inclusion functor $\mathbf{LCB}(B)\to \mathbf{Top}/B$ which is faithful but not full (not every map between locally constant spaces over $B$ is locally constant).

Covering spaces over $B$ are a special case of locally constant $G$-bundles over $B$. It can be proved that every map of covering spaces is locally constant. The natural inclusion functor is thus full and faithful
\[
\mathbf{Cov}(B) \subset \mathbf{LCB}(B).
\]
Trivial bundles are also examples of locally constant bundles, but not every map between trivial bundles is locally constant.

\begin{proposition}\label{equivcont}
The functors $F:\mathbf{LCB}(B)\to G\textrm{-}\mathbf{Space}$ sending $p:E\to B$ to $p^{-1}(b_0)$ 
and $U:G\textrm{-}\mathbf{Space}\to \mathbf{LCB}(B)$ sending $X$ to $X\times_G \widetilde B\to B$
are inverse equivalences of categories.
\end{proposition}
\begin{proof}
Similar to proposition \ref{equivdiscrete}.
\end{proof}

Summarizing, we have the following diagram of categories
\[
\xymatrix{ 
\mathbf{Cov}(B) 
\ar[r]^-\simeq
\ar@{^{(}->}[d]
& \mbox{$G$-\textbf{Set}}
\ar@{^{(}->}[d]\\ 
\textbf{LCB}(B) 
\ar[r]^-\simeq
& \mbox{$G$-\textbf{Space}}
\\
}
\]
where the horizontal arrows are equivalences of categories and the vertical ones fully faithful embeddings.

\bigskip

We now recall some notions that will be needed in the next section.
If $p:E\to B$ is a bundle, a {\em sub-bundle} of $E$ is defined as a subspace $F\subset E$ such that the map $F\to B$ is still a locally trivial map. If $E\to B$ is locally constant, not every sub-bundle is locally constant.

For $i=1,2,3$, let $p_i:E_i \rightarrow B$ be three bundles over $B$ and let $f_1:E_1\to E_3$ and $f_2:E_2\to E_3$ be two bundle maps. The fibre product of $f_1$ and $f_2$, noted $E_1\times_{E_3}E_2$ can be defined as follows.
Consider $f_1 \times f_2:E_1 \times E_2 \to E_3 \times E_3$, with each product topologized with the product topology.  Then $f_1 \times f_2$ has the formula $(f_1 \times f_2)(e_1, e_2) = (f_1(e_1), f_2(e_2))$.
Define
\[ E_1 \times_{E_3} E_2 = \{ (e_1, e_2) \in E_1 \times E_2 | f_1(e_1) = f_2(e_2) \},
\]
and topologize it using the subspace topology. There is an obvious map $E_1 \times_{E_3} E_2\to E_3$ and by composition with $p_3$ a map $p_4:E_1 \times_{E_3} E_2\to B$.

In the particular case where $E_3=B$, $E_1\times_BE_2$ is simply called the product of $E_1$ and $E_2$.

\begin{lemma}\label{pbLCB}
If the $E_i$ are bundles over $B$, then $E_1 \times_{E_3} E_2\to E_3$ is a bundle over $E_3$ and $E_1 \times_{E_3} E_2\to B$ is a bundle over $B$.
Moreover, if the $E_i$ are locally constant bundles so is $E_1 \times_{E_3} E_2\to B$.
\end{lemma}
\begin{proof}
Let us assume that the $E_i$ are bundles, i.e. locally trivial maps, but such maps are stable by pull-backs and composition.
If now the $E_i$ are locally constant bundles, we need to prove that $E_1 \times_{E_3} E_2\to B$ is trivial when pulled-back on $\widetilde B$. For this we use that $\widetilde {E_1 \times_{E_3} E_2} \simeq \widetilde E_1 \times_{\widetilde E_3} \widetilde E_2$ and as all of $\widetilde E_i$ are trivial bundles, so is their fibre product.
\end{proof}

The {\em image} of a morphism $f:E_1\to E_2$ of bundles over $B$ is defined as the set $im(f)=\{y\in E_2|\exists x\in E_1,y=f(x)\}$ topologized with the subspace topology. In general, the image of a map of bundles is not a bundle, but in the case of locally constant maps between locally constant bundles the image is a locally constant bundle: if $\widetilde f:\widetilde E_1\to \widetilde E_2:(x,\tilde b)\mapsto (g(x),\tilde b)$ is the map induced by $f$ over $\widetilde B$, then $im(\widetilde f) = \widetilde{im(f)} = im(g)\times \widetilde B$.

The diagonal $\Delta_E$ of a bundle $p:E\to B$ is defined as the image of the map $E\to E\times_BE:x\mapsto(x,x)$, it is isomorphic to $E\to B$.

All these notions restrict to covering spaces : any locally trivial sub-bundle of a covering space is a covering space and the fibre product of covering spaces is still a covering space, etc.

\medskip

Let $X$ be a $G$-space, a subspace $Y\subset X$ is call a sub-$G$-space if it is stable by the action of $G$.

For $i=1,2,3$, let $X_i$ be three $G$-spaces, and let $f_1:X_1\to X_3$ and $f_2:X_2\to X_3$ be two bundle maps. The fibre product of $f_1$ and $f_2$ is  denoted $X_1\times_{X_3}X_2$ and defined as follows. The underlying set is the fibre product of the underlying sets, i.e. the set of pairs $(x_1,x_2)$ such that $f_1(x_1)=f_2(x_2)$ and the action is defined by $(x_1,x_2) \cdot g$ = $(x_1 \cdot g,x_2 \cdot g)$. In the case where $E_3$ is the $G$-set $*$ (see Example \ref{triv}), the fibre product is simply called the product and is denoted $X_1\times X_2$.

The {\em image} of a morphism $f:X_1\to X_2$ of $G$-spaces is defined as the set $im(f)=\{y\in X_2|\exists x\in X_1,y=f(x)\}$ topologized with the subspace topology. It is clear that the action of $G$ on $X_2$ restricts to $im(f)$.

The diagonal $\Delta_X$ of a $G$-space $X$ is defined as the image of the map $X\to X\times X:x\mapsto(x,x)$ with the obvious action. It is isomorphic to $X$.

All these notions restricts to $G$-sets : any sub-$G$-space of a $G$-set is a $G$-set and the fibre product of $G$-sets is still a $G$-set, etc.

\medskip

The notions of subobjects, products, fibre products, image and diagonal can be defined in purely categorical terms and are thus preserved under equivalences of categories \cite{MLane}. As a consequence, all previous notions on locally constant bundles over $B$ correspond to the same notion of $G$-spaces. For example, there is a bijection between locally trivial sub-bundles of a locally trivial bundle $E\to B$ and sub-$G$-spaces of the corresponding fibre $X$.


\section{Ordered structures}

We start by recalling some facts about relations and total orders on a given set.
If $R\subset X\times X$ is a relation on a set $X$, the {\em opposite relation } $R^{op}$ is defined as
$R^{op} = \{(x,y)| yRx\}$, or equivalently as the inclusion $R\to X\times X\overset{\sigma}{\simeq} X\times X$ where the map $\sigma$ is the permutation of the two factors.

For $i=1,2$, if $R_i\subset X\times X$ are two relations on $X$, their {\em composition} $R_2\circ R_1$, is defined as $R_2\circ R_1 = \{(x,z)| \exists z\in X, xR_1y\ \textrm{and}\ yR_2z \}$.
Equivalently, if $s_i$ and $t_i$ are respectively the first and second projections of $R_i$ on $X$, the fibre product of $t_1:R_1\to X$ and $s_2:R_2\to X$ is the set $R_1\times_XR_2:\{(x,y,z)| xR_1y\ \textrm{and}\ yR_2z \}$.
The maps $s_1$ and $t_2$ define a map $R_1\times_XR_2\to X\times X$ and $R_2\circ R_1$ can be defined as the image of this map.

\medskip

Recall that a {\em strict total order} on a set $X$ is a relation $<$ such that:
\begin{enumerate}
\item (Transitivity) ($x<y$ and $y<z$) $\Rightarrow$ $x<z$,
\item (Irreflexivity) not $x< x$,
\item (Antisymmetry) not ($x<y$ and $y<x$),
\item (Totality) $x\not=y \Rightarrow$ ($x<y$ or $y<x$).
\end{enumerate}
If $(X_1,<_1)$ and $(X_2,<_2)$ are two strictly totally ordered sets, a map $f:X_1\to X_2$ is {\em order-preserving} if $x<_1y\Rightarrow f(x)<_2f(y)$.

\medskip

Let $R$ be the subobject of $X\times X$ corresponding to a strict total order $<$. Using the previous considerations, these axioms can be written as the following conditions on subobjects of $X\times X$:
\begin{enumerate}
\item (Transitivity) $R\circ R\subset R$,
\item (Irreflexivity) $\Delta_X \cap R = \emptyset$ (where $\Delta_X\subset X\times X$ is the diagonal),
\item (Antisymmetry) $R \cap R^{op} = \emptyset$,
\item (Totality) $R \cup R^{op} = X\times X \setminus \Delta_X$.
\end{enumerate}
If $(X_1,<_1)$ and $(X_2,<_2)$ are two strictly totally ordered sets, corresponding to subobjects $R_1\subset X_1\times X_1$ and $R_2\subset X_2\times X_2$, a map $f:X_1\to X_2$ is order-preserving if and only if $f\times f$ sends $R_1$ into $R_2$.

\begin{example}\label{exR}
Let $(\R,<)$ be the set of real numbers with its canonical strict total order, the corresponding subobject $R_<\subset \R^2$ is the set $R_<=\{(x,y)|x<y\}$.
\end{example}

\bigskip
An important remark for the sequel is that these definitions make sense not only for sets but in any category $\mathbf C$ where subobjects, products and fibre products exist. In particular these definitions make sense in the categories $\mathbf{Cov}(B)$, $\textbf{LCB}(B)$, $G$-$\textbf{Set}$ and $G$-$\textbf{Space}$.

Let $\mathbf C$ be one of the categories $\mathbf{Cov}(B)$, $\textbf{LCB}(B)$, $G$-$\textbf{Set}$ and $G$-$\textbf{Space}$.
A {\em relation} on an object $X\in \mathbf C$ is a subobject $R\subset X\times X$.
The {\em opposite} of a relation $R$ is the relation $R^{op}$ defined by $R\subset X\times X\overset{\sigma}{\simeq}X\times X$, where $\sigma$ is the permutation of the two factors.
For two relations $R_1\subset X\times X$ and $R_2\subset X\times X$, if $s_i$ and $t_i$ are respectively the first and second projections of $R_i$ on $X$, the fibre product of $t_1:R_1\to X$ and $s_2:R_2\to X$ is the set $R_1\times_XR_2:\{(x,y,z)| xR_1y\ \textrm{and}\ yR_2z \}$. The {\em composition} $R_2\circ R_1$ is defined as the image of the map $R_1\times_XR_2\to X\times X$ induced by $s_1$ and $t_2$.

\begin{definition}\label{deforder}
A relation $R\subset X\times X$ on an object of $\mathbf{C}$ is said to be a {\em strict total order} if it satisfies the following properties
\begin{enumerate}
\item (Transitivity) $R\circ R\subset R$,
\item (Irreflexivity) $\Delta_X \cap R = \emptyset$,
\item (Antisymmetry) $R \cap R^{op} = \emptyset$,
\item (Totality) $R \cup R^{op} = X\times X \setminus \Delta_X$.
\end{enumerate}
\end{definition}

In particular, for such an order $R$, we have $X\times X = R\sqcup R^{op}\sqcup \Delta_X$ where $\sqcup$ denotes the disjoint union.

\begin{definition} \label{ordered bundle}
A locally constant bundle $p: E \rightarrow B$ is {\em orderable} if there exists a locally constant sub-bundle $R \subset E \times_B E$ such that $R$ satisfies the axioms of a strict total order.

Suppose that $E_1$ and $E_2$ are ordered bundles with order relations $R_1 \subset E_1 \times_B E_1$ and $R_2 \subset E_2 \times_B E_2$. We call a map $f:E_1 \rightarrow E_2$ {\em order-preserving} if the map $f \times f:E_1 \times_B E_1 \to E_2 \times_B E_2$ satisfies $(f \times f)(R_1) \subset R_2$.
\end{definition}

\begin{lemma}
If $R \subset E \times_B E$ is a strict total order on some bundle $E\to B$, and if $B'\to B$ is any continous map, the relation $R'\subset E'\times_{B'}E'$ where $E'=E\times_BB'$ and $R'=R\times_{B\times B}B'\times B'$ is a strict total order on $E'\to B'$.
\end{lemma}

This lemma applies in particular in the case where $B'$ is a point, and says that a strict total order on $E\to B$ gives a strict total order in every fibre of $E\to B$. In the same way a bundle can be thought of as the continous family of its fibres, a strict total order on a bunde can be thought of as a continous family of strict total orders.

\medskip

We will not use the following but it gives a way to construct ordered bundles.

\begin{lemma}\label{restrictorder}\label{pullback bundle order}
If $f:E_1\to E_2$ is an injective locally constant map of locally constant bundles and if $R_2\subset E_2\times_BE_2$ is a strict total order on $E_2$, then $R_1=(f\times f)^{-1}(R_2)$ is a strict total order on $E_1$.
\end{lemma}
\begin{proof}
$R_1$ is a locally trivial sub-bundle of $E_1\times_BE_1$ by Lemma \ref{pbLCB}. Properties (2)-(5) of $R$ are easy to check. Note that injectivity of $f$ is required only to prove (5).
 \end{proof}

\begin{example}
Recall from Example \ref{exR} $R_<=\{(x,y)|x<y\}\subset \R^2$ defining the canonical order on $\R$.
Let $\R_B=\R\times B\to B$ be the trivial line bundle, then $\R_B\times_B\R_B\simeq \R\times \R\times B\to B$ is again a trivial bundle.  The trivial sub-bundle $R_<\times B\to B$ is a locally constant sub-bundle of $\R_B\times_B\R_B\to B$ which defines a strict total order on $\R_B$. Fibrewise, this order is nothing more than the canonical order on $\R$.
\end{example}

\begin{example}
If $\alpha:G\to \mathrm{Homeo}_+(\R)$ is a group morphism, $G$ acts on $\R$ and $R_<$. Then, there is a canonical strict total order $R_\alpha$ on $\R\times_G\widetilde B$ given by $R_\alpha=R_<\times_G\widetilde B$.
\end{example}


\bigskip

We can apply Definition \ref{deforder} to $G$-spaces. 

\begin{definition}
\label{ordered gset}
A $G$-space $X$ is {\em orderable} if there exists a $G$-invariant subspace $R \subset X \times X$ satisfying the axioms of a strict total order.

Suppose that $X_1$ and $X_2$ are ordered bundles with order relations $R_1 \subset X_1 \times X_1$ and $R_2 \subset X_2 \times X_2$. A map $f:X_1 \rightarrow X_2$ {\em order-preserving} if the map $f \times f:X_1 \times X_1 \to X_2 \times X_2$ satisfies $(f \times f)(R_1) \subset R_2$.
\end{definition}

In the case of a $G$-space, a relation $R$ provides a strict total ordering $<$ of the underlying space $X$.  Moreover, since $R$ is $G$-invariant, the ordering is invariant under the right action of $G$: for all $x, y \in X$, $x<y$ implies $x \cdot g < y \cdot g$.

\begin{example}[Right-orderable groups]
Recall the right $G$-set $G_r$ from Example \ref{exGright}. 
The group $G$ is right-orderable if and only if $G_r$ is orderable as a $G$-set.  If $G$ has ordering $<$, there is a corresponding $R \subset G_r \times G_r$ defined by $(g, h) \in R$ if and only if $g<h$.   Conversely, given $R \subset G_r \times G_r$ since $R$ is $G$-invariant the same rule defines a right-invariant ordering $<$ of the elements of $G$.
\end{example}

\begin{example}[Bi-orderable groups]
\label{ex:bi-orderable groups}
Recall the right $G\times G$-set $_{\ell}G_r$ from Example \ref{exGbi}. 
The group $G$ is bi-orderable if and only if $_{\ell}G_r$ is an orderable $(G \times G)$-set.  If $G$ is bi-orderable with ordering $<$, define $R \subset \protect{ _{\ell}G_r} \times  \protect{ _{\ell}G_r}$ according to the rule $(g,h) \in R$ if and only if $g<h$.  Since the ordering of $G$ is bi-invariant, $R$ is $(G \times G)$-invariant.  

Conversely, if $\protect{ _{\ell}G_r}$ is orderable, then  $R \subset \protect{ _{\ell}G_r} \times  \protect{ _{\ell}G_r}$ defines a bi-ordering of $G$ by the same rule: $g<h$ if and only if $(g,h) \in R$.  Two-sided invariance of this ordering follows from $(G \times G)$-invariance of $R$.  For if $g<h$, then $(g, h) \in R$ and for all $f \in G$ we have
\[ (g,h) \cdot (f^{-1}, id) \in R \iff (fg, fh) \in R \iff fg<fh. 
\]
The calculation to show right-invariance is similar.
\end{example}

\begin{proposition}
\label{equivalence of definitions}
Let $p:E \rightarrow B$ be a locally constant bundle and $X$ its fibre at $b_0$. Then $p:E \rightarrow B$ is orderable as a bundle if and only if $X$ is orderable as a $G$-space.
\end{proposition}
\begin{proof}
Each of the properties (1)-(5) of ordered bundles and ordered $G$-spaces is preserved under equivalence of categories, since each of properties (1)-(5) of the subobject $R$ can be defined diagrammatically.
\end{proof}

We also point out that order-preserving maps of bundles are sent to order preserving maps of $G$-sets by the functor $F$, and reciprocally for the functor $U$.

\section{Results}

We begin with some well-known lemmas, for example see \cite[Proposition 2.1]{Nav}.

\begin{lemma}
\label{include in R}
If $(X, <)$ is a countable strictly totally ordered set, then there exists an order-preserving embedding $X \hookrightarrow \mathbb{R}$.
\end{lemma}
\begin{proof}
Given an enumeration of $X = \{ x_0, x_1, \ldots\}$, construct an order preserving map $t:X \rightarrow \mathbb{R}$ as follows.  Set $t(x_0)=0$, and if $t(x_0), \ldots, t(x_i)$ have already been defined, there are three cases to consider when defining $t(x_{i+1})$.  If $x_{i+1}$ is either larger or smaller than all of $\{x_0, \ldots, x_i\}$, set
\begin{displaymath}
   t(x_{i+1}) = \left\{
     \begin{array}{ll}
       \mathrm{max}\{t(x_0), \ldots , t(x_i) \} +1 & \mbox{ if } x_{i+1} > \mathrm{max}\{x_0, \ldots , x_i \} \\
       \mathrm{min}\{t(x_0), \ldots , t(x_i) \} -1 & \mbox{ if } x_{i+1} < \mathrm{min}\{x_0, \ldots , x_i \} \\

     \end{array}
   \right.
\end{displaymath} 
If neither of these cases hold, then there exist $j, k \in \{ 0, \ldots, i \}$ such that $x_j < x_{i+1} < x_k$ and there is no $i_0 \in \{ 0, \ldots ,i \}$ such that $g_j < g_{i_0} < g_k$.  In this case set
\[t(x_{i+1})= \cfrac{t(x_j)+t(x_k)}{2}
\]
This map preserves the order by construction and is injective.
\end{proof}

\begin{lemma}\cite[{\em cf.} Lemma 1.10]{Fa}
\label{lemma:map to R}
If $X$ is an orderable $G$-set and $t :X \rightarrow \mathbb{R}$ an order-preserving embedding, then there exists a right action of $G$ on $\mathbb{R}$ by order-preserving homeomorphisms making $t$ a $G$-equivariant map of $G$-sets.
\end{lemma}

\begin{proof} 

 Outside of the  interval 
$$I = (  \mathrm{inf}( t( X)),  \mathrm{sup}( t( X) ))$$
define the action of $G$ to be the trivial action.

Inside the interval $I$ we proceed as follows.  Evidently $G$ acts in an order-preserving way on $t(X)$ according to the rule $(t(e)) \cdot g = t(e \cdot g)$, we extend this action to the closure $\overline{t(X)}$ in such a way that every $g \in G$ acts continuously on $\overline{t(X)}$ .  Specifically, for $x \in \overline{t(X)} \setminus X$ choose a monotone sequence of points $t(x_i)$ converging to $x$ from above (or below).   Assume that the sequence is monotone increasing (the case of monotone decreasing is similar), since $x \in I$ there exists $x' \in X$ with $t(x') > x$.   For every element $g \in G$, define the action of $g$ on $x$ as follows:
\[ x \cdot g = \lim_{i \to \infty} t(x_i \cdot g)
\]
This limit exists because the sequence $\{ t(e_i \cdot g) \}$ is monotone increasing and bounded above by $t(x' \cdot g)$. It is not hard to see that this definition is independent of our choice of sequence $\{ t(x_i) \}$.

Now $\overline{t(X)}$ is closed, and the complement $\mathbb{R} \setminus \overline{t(X)}$ is a union of open intervals with the $G$-action defined at their endpoints.  Extend the $G$-action affinely across all intervals, this defines the required action of $G$ on $\mathbb{R}$.
\end{proof}

The next lemma allows us to understand those bundles with structure group $\mathrm{Homeo}_+(\mathbb{R})$.

\begin{lemma}\cite[Theorem 1.1.1]{Ham}
\label{lemma:contractible}
Equipped with the compact open topology, the group $\mathrm{Homeo}_+(\mathbb{R})$ is a contractible space.
\end{lemma}
\begin{proof}
For every $f \in \mathrm{Homeo}_+(\mathbb{R})$ and $t \in [0, 1]$ define
\[ H(f, t) = (1-t)f(x) +tx.
\]
If $f$ is order-preserving, then $ H(f, t)$ is order-preserving for all $t \in [0,1]$, and $H(f,0) = f$ while $H(f,1) = id$.  Therefore $\mathrm{Homeo}_+(\mathbb{R})$ is contractible.
\end{proof}

In the following proposition and theorem, we will use the notation $p:E \rightarrow B$ to denote a bundle that is isomorphic to 
\[ U(X) = X\times_G\widetilde B \to B 
\]
for some $G$-set $X$.
In particular, the fibre of $E$ over $b_0$ is isomorphic to $X$.

\begin{proposition}
\label{prop:two maps}
For an ordered $G$-set $X$, the following are equivalent:
\begin{enumerate}
\item There exists an order-preserving injective map $t: X \hookrightarrow \mathbb{R}$.
\item There exists an order-preserving injective map of bundles $f:E \hookrightarrow \mathbb{R} \times B$.  Here, $ \mathbb{R} \times B$ is equipped with the natural order of $\mathbb{R}$ on the fibres.
\end{enumerate}
\end{proposition}
\begin{proof}
By Lemma \ref{lemma:map to R}, $\mathbb{R}$ can be equipped with a $G$-action so that $t$ is an equivariant map.  Thus the map $t:X \rightarrow \mathbb{R}$ can be considered as a map of $G$-spaces.  

Correspondingly there is an order-preserving locally constant map $f'$ of locally constant bundles
\[
\xymatrix{ 
E 
\ar[rr]^{f'} 
\ar[dr]_{p} 
&& \R\times_G\widetilde{B}
\ar[dl]^{p' } \\ 
& B. }
\]
The structure group of $\R\times_G\widetilde{B}$ is $\mathrm{Homeo}_+(\mathbb{R})$, which is contractible by Lemma \ref{lemma:contractible}. So the bundle is trivial \cite[Corollary 12.3]{Steenrod}.  
There exists an isomorphism $\R\times_G\widetilde{B}\simeq \R\times B$ (which is not locally constant) and this isomorphism is given by a function with values in the structure group $\mathrm{Homeo}_+(\mathbb{R})$, hence it preserves the order.
By composition, we deduced an injective map $f:E \rightarrow \mathbb{R} \times B$, which is still order preserving (but no longer locally constant).

On the other hand, suppose we an order-preserving injective map $f: E \rightarrow \mathbb{R} \times B$, say $f= (f_1, f_2)$.  Thus we have a map $f_1: E \rightarrow \mathbb{R}$ which respects the order on each fibre, and by restriction to the fibre $X$ we get an order-preserving map $t: X \hookrightarrow \mathbb{R}$.
\end{proof}

We arrive now at our main theorem.

\begin{theorem}
\label{thm:main result}
For a countable $G$-set $X$, the following are equivalent.
\begin{enumerate}
\item $X$ is orderable.
\item $p:E \rightarrow B$ is orderable as a locally constant bundle.
\item There exists an embedding $E \rightarrow \R_B$ such that 
\[\xymatrix{ 
E 
\ar[rr]^{f} 
\ar[dr]_{p} 
&& \R_B
\ar[dl]^{\pi_2 } \\ 
& B }
\]
commutes.
\end{enumerate}
\end{theorem}
\begin{proof}
The equivalence of statements (1) and (2) was already observed in Proposition \ref{equivalence of definitions}.
By Lemma \ref{include in R}, if $X$ is ordered and countable then there exists an injective, order-preserving map $X \rightarrow \mathbb{R}$.  By Proposition \ref{prop:two maps} There exists an embedding $E \hookrightarrow \mathbb{R} \times B$.  Thus (1) implies (3).  

We will finish the proof by proving that (3) implies (2).
Let $R'\subset E\times_BE$ be the pull back of the canonical order relation $R\times B\subset\R_B\times_B\R_B$ by the map $f:E\to \R_B$. Then $R'$ satisfies all the axioms of a strict total order, but because $f$ is not a locally constant map, we cannot use Lemma \ref{pbLCB} to deduce that $R'$ is a sub-covering space of $E\times_BE$.

To prove that $R'$ is indeed a sub-covering, it is sufficient to prove that it is a union of path connected components of $E\times_BE$.
Remark that a bundle map $f:E\to \R_B=\R\times B$ over $B$ is the same thing as a map $g:E\to \R$. From such a map $g$ the partition $R_<\sqcup R_<^{op}\sqcup \Delta_\R$ of $\R^2$ is pulled back by $g$ to a partition $R'\sqcup (R')^{op} \sqcup D$ of $E\times_BE$ where $D=(g\times g)^{-1}(\Delta_\R)$, but because $f$ is injective, $D=\emptyset$.

So $R'$ is a sub-covering if we prove that no path starting in $R'$ can arrive in $(R')^{op}$.
Let us assume that we have such a path; by composing with $g$, we obtain a path in $\R^2$ starting in $R_<$ and ending in $R_<^{op}$, i.e. a family of numbers $x(t)$ and $y(t)$ such that  $x(0)<y(0)$ and $x(1)>y(1)$.
Then, the topology of $\R$ is such that we must have $x(t)=y(t)$ for some $t$, but this would give two elements $e_1$ and $e_2$ such that $f(e_1)=f(e_1)$ and contradict the injectivity of $f$.

%

\end{proof}

\begin{corollary}\cite[{\em cf.} Theorem 2.3]{Fa}
\label{cor:farrell}
Suppose $G$ is a countable group and $p:E \rightarrow B$ a covering space.

 The right cosets $p_*(\pi_1(E, e_0))\backslash G$ form an orderable right $G$-set if and only if 
there exists a bundle embedding $E \hookrightarrow \mathbb{R} \times B$ such that 

\[\xymatrix{ 
E 
\ar[rr]^{f} 
\ar[dr]_{p} 
&& \mathbb{R} \times B
\ar[dl]^{\pi_2 } \\ 
& B }
\]
commutes. In particular, $G$ is right orderable if and only if there is a bundle embedding $\widetilde{B} \hookrightarrow \mathbb{R} \times B$.
\end{corollary}
\begin{proof}
The fibre of the covering $p:E \rightarrow B$ can be naturally identified with the countable $G$-set $p_*(\pi_1(E, e_0))\backslash G$, which is isomorphic to $G$ in the case that $E = \widetilde{B}$.  The result now follows from the equivalence of (1) and (3) in Theorem \ref{thm:main result}.
\end{proof}

The generality of Theorem \ref{thm:main result} allows for an analysis of other ordered structures.
Recall the path space $P(B) \rightarrow B$ from Example \ref{exGbi}, whose fibre is the $G\times G$-set  $_{\ell}G_r$.

\begin{corollary}
A countable group $G$ is bi-orderable if and only if there exists an embedding $P(B) \rightarrow \mathbb{R} \times B \times B$ such that 

\[\xymatrix{ 
P(B)
\ar[rr]^{f} 
\ar[dr]_{p} 
&& \mathbb{R} \times (B \times B)
\ar[dl]^{\pi_2 } \\ 
& B \times B }
\]
commutes.
\end{corollary}
\begin{proof}
The group $G$ is bi-orderable if and only if the fibre of $p:P(B) \rightarrow B \times B$ is an ordered $G \times G$-set.  By Theorem \ref{thm:main result} this happens if and only if there is an embedding $P(B) \rightarrow \mathbb{R} \times B \times B$ making the above diagram commute.
\end{proof}

\bibliographystyle{plain}
\bibliography{ob}

\end{document}